\documentclass{amsart}
\usepackage{amsmath}
\usepackage{amssymb}
\usepackage{amsthm}
\usepackage{enumerate}
\usepackage[dvips]{graphicx}
\theoremstyle{definition}
\newtheorem{definition}{Definition}[section]
\theoremstyle{plain}
\newtheorem{lemma}[definition]{Lemma}
\newtheorem{theorem}[definition]{Theorem}
\newtheorem{proposition}[definition]{Proposition}
\newtheorem{corollary}[definition]{Corollary}
\theoremstyle{remark}

\newcommand{\myds}{\operatorname{D}_\Sigma}
\newcommand{\myint}{\operatorname{int}}
\newcommand{\mycl}{\operatorname{cl}}
\newcommand{\mylc}{\operatorname{lc}}
\newcommand{\myfdim}{\operatorname{f.dim}}
\newcommand{\myim}{\operatorname{Im}}

\makeatletter
\@namedef{subjclassname@2010}{
  \textup{2010} Mathematics Subject Classification}
\makeatother

\begin{document}

\title[locally o-minimal  open core]{locally o-minimal open core}
\author[M. Fujita]{Masato Fujita}
\address{Department of Liberal Arts,
Japan Coast Guard Academy,
5-1 Wakaba-cho, Kure, Hiroshima 737-8512, Japan}
\email{fujita.masato.p34@kyoto-u.jp}

\begin{abstract}
We demonstrate that the open core of a definably complete expansion of a densely linearly ordered abelian group is locally o-minimal if and only if any definable closed subset of $R$ is either discrete or contains a nonempty open interval.
Here, the notation $R$ denotes the universe of the original structure.
\end{abstract}

\subjclass[2010]{Primary 03C64}

\keywords{open core, locally o-minimal structure}

\maketitle

\section{Introduction}\label{sec:intro}
The open core of an expansion of a linear order is its reduct generated by its definable open sets. 
The notion of open core is first introduced in \cite{MS} and Dolich et al. gave a sufficient condition for the structure having an o-minimal open core in \cite{DMS}.

Local o-minimality was first proposed in \cite{TV} as a local variant of o-minimality and it has been studied in \cite{F, Fuji, Fuji3, Fuji5, KTTT}.
To the best of the author's knowledge, Fornasiero first investigated the structures having locally o-minimal open cores.
He gave necessary and sufficient conditions for a definably complete structure having a locally o-minimal open core in \cite{F}.
He assumed that the structure is an expansion of an ordered field.
The author considered the problem under a relaxed algebraic assumption.
He provided a sufficient condition for a definably complete expansion of a densely linearly ordered abelian group having a uniformly locally o-minimal open core of the second kind in \cite{Fuji4}.
Definably complete uniformly locally o-minimal structures of the second kind satisfy the additional condition which general locally o-minimal structures do not satisfy. 
They admit local definable cell decomposition though general locally o-minimal structures do not.
See \cite{Fuji} for more details on uniformly locally o-minimal structures of the second kind.

The purpose of this paper is to find a necessary and sufficient condition for a definably complete structure having a locally o-minimal open core under the same assumption employed in \cite{Fuji4}.
The following is the main theorem of this paper.

\begin{theorem}\label{thm:main}
Let $\mathcal R=(R,<,+,0,\ldots)$ be a definably complete expansion of a densely linearly ordered abelian group.
The following are equivalent:
\begin{enumerate}
\item[(1)] Any definable closed subset of $R$ is either discrete or contains a nonempty open interval.
\item[(2)] The open core of $\mathcal R$ is locally o-minimal.
\end{enumerate}
Furthermore, sets definable in the open core are constructible when one of (and both of) the above equivalent conditions is satisfied.
Recall that a {constructible} set is a finite boolean combination of open sets.
\end{theorem}

Tools and techniques for proving this theorem were already developed in the previous studies \cite{DMS, F, Fuji3, Fuji4}.
In Section \ref{sec:preliminary}, we recall the definitions and results given in the previous works.
Two key lemmas are demonstrated in Section \ref{sec:key} mimicing the arguments in \cite{Fuji3, Fuji4}.
We complete the proof of the theorem in Section \ref{sec:proof}.

We introduce the terms and notations used in this paper.
The term `definable' means `definable in the given structure with parameters' in this paper.
For any set $X \subset R^{m+n}$ definable in a structure $\mathcal R=(R,\ldots)$ and for any $x \in R^m$, the notation $X_x$ denotes the fiber defined as $\{y \in R^n\;|\; (x,y) \in X\}$.
For a linearly ordered structure $\mathcal R=(R,<,\ldots)$, an open interval is a definable set of the form $\{x \in R\;|\; a < x < b\}$ for some $a,b \in R$.
It is denoted by $(a,b)$ in this paper.
An open box in $R^n$ is the direct product of $n$ open intervals.
A \textit{CBD} set is a closed, bounded and definable set.
Let $A$ be a subset of a topological space.
The notations $\myint(A)$ and $\mycl(A)$ denote the interior and the closure of the set $A$, respectively.

\section{Preliminary}\label{sec:preliminary}
We recall the definitions and the assertions introduced in the previous studies.

\begin{definition}[\cite{M,TV}]
An expansion of a densely linearly ordered set without endpoints $\mathcal R=(R,<,\ldots)$ is \textit{definably complete} if any definable subset $X$ of $R$ has the supremum and  infimum in $R \cup \{\pm \infty\}$.
The structure $\mathcal R$ is \textit{locally o-minimal} if, for every definable subset $X$ of $R$ and for every point $a\in R$, there exists an open interval $I$ containing the point $a$ such that $X \cap I$ is  a finite union of points and open intervals.
\end{definition}

Dolich et al.\ used $\myds$-sets in \cite{DMS}.
The author also used them in \cite{Fuji4}.
They play an important role also in this paper.
\begin{definition}[$\myds$-sets]
Consider an expansion of a densely linearly ordered abelian group $\mathcal R=(R,<,+,0\ldots)$.

A \textit{parameterized family} of definable sets $\{X{\langle x \rangle}\}_{x \in S}$ is the family of the fibers of a definable set; that is, there exists a definable set $\mathcal X$ with $X{\langle x \rangle}=\mathcal X_x$ for all $x$ in a definable set $S$.
A parameterized family $\{X{\langle r,s\rangle}\}_{r>0,s>0}$ of CBD subsets $X{\langle r,s\rangle}$ of $R^n$ is called a \textit{$\myds$-family} if $X{\langle r,s\rangle} \subseteq X{\langle r',s\rangle}$ and  $X{\langle r,s'\rangle } \subseteq X{\langle r,s\rangle}$ whenever $r < r'$ and $s < s'$.
Note that $X{\langle r,s\rangle}$ is not necessarily strictly contained in $X{\langle r',s\rangle}$.
It is the same for the inclusion $X{\langle r,s'\rangle } \subseteq X{\langle r,s\rangle}$.
A definable subset $X$ of $R^n$ is a \textit{$\myds$-set} if $X = \displaystyle\bigcup_{r>0,s>0} X{\langle r,s \rangle}$ for some $\myds$-family $\{X{\langle r,s\rangle}\}_{r>0,s>0}$.
\end{definition}

The following two lemmas are found in \cite{DMS,M}.
We use Lemma \ref{lem:quo} here and there without notice in this paper.
\begin{lemma}\label{lem:quo}
Consider a definably complete expansion of a densely linearly ordered abelian group $\mathcal R$.
The following assertions are true:
\begin{enumerate}[(1)]
\item The projection image of a $\myds$-set is $\myds$.
\item Fibers, finite unions and finite intersections of $\myds$-sets are $\myds$.
\item Every constructible definable set is $\myds$.
\end{enumerate}
\end{lemma}
\begin{proof}
(1) Immediate from \cite[Lemma 1.7]{M}.\ 
(2) \cite[1.9(1)]{DMS}.\ 
(3) \cite[1.10(1)]{DMS}.
\end{proof}

\begin{lemma}\label{lem:interior0}
Let $\mathcal R=(R,<,+,0,\ldots)$ be a definably complete expansion of a densely linearly ordered abelian group.
A CBD set $X \subset R^{n+1}$ has a nonempty interior if the CBD set 
\begin{center}
$\{x \in R^n\;|\; X_x \text{ contains a closed interval of length }s\}$
\end{center}
has a nonempty interior for some $s>0$.
\end{lemma}
\begin{proof}
\cite[2.8(2)]{DMS}
\end{proof}

We give a definition of dimension of a $\myds$-set.
\begin{definition}[Dimension]
Let $\mathcal R=(R,<,\ldots)$ be an expansion of a densely linearly ordered structure.
Consider a $\myds$-subset $X$ of $R^n$ and a point $x \in R^n$.
The \textit{dimension of $X$} is defined as follows:
\begin{itemize}
\item $\dim X=-\infty$ if $X$ is an empty set.
\item Otherwise, $\dim X$ is the supremum of nonnegative integers $d$ such that the image $\pi(X)$ has a nonempty interior for some coordinate projection $\pi:R^n \rightarrow R^d$.
\end{itemize} 
The \textit{full dimension} $\myfdim(X)$ of $X$ is the pair $\langle d ,k \rangle$, where $d=\dim X$ and $k$ is the number of coordinate spaces of dimension $d$ such that the images of $X$ under the projections of $R^n$ onto the coordinate spaces have nonempty interiors.
\end{definition}

\begin{definition}
For a set $X$, a family $\mathcal F$ of subsets of $X$ is called a \textit{filtered collection} if, for any $B_1, B_2 \in \mathcal F$, there exists $B_3 \in \mathcal F$ with $B_3 \subseteq B_1 \cap B_2$. 
A parameterized family $\{S \langle t \rangle\}_{t \in T}$ of definable subsets of $X$ is a \textit{definable filtered collection} if it is a filtered collection.
\end{definition}

\begin{proposition}\label{prop:def_compact}
Consider a definably complete expansion of a dense linear order without endpoints.
Every definable filtered collection of nonempty definable closed subsets of a closed, bounded and definable set has a nonempty intersection.
\end{proposition}
\begin{proof}
The literally same proof as that for o-minimal structures in \cite[Section 8.4]{J} works. 
\end{proof}

The following lemma is found in \cite[Lemma 2.3]{Fuji5}.

\begin{lemma}\label{lem:local}
Consider a definably complete structure $\mathcal R=(R,<,\ldots)$.
The following are equivalent:
\begin{enumerate}
\item[(1)] The structure $\mathcal R$ is a locally o-minimal structure.
\item[(2)] Any definable set in $R$ has a nonempty interior or it is closed and discrete. 
\end{enumerate}
\end{lemma}

\section{Key lemmas}\label{sec:key}
We demonstrate key lemmas and their corollaries in this section.

\subsection{First key lemma}\label{sec:keysub}
We first introduce the following technical definition.
\begin{definition}\label{def:technical}
Let $\mathcal R=(R,<,+,0,\ldots)$ be a definably complete expansion of a densely linearly ordered abelian group.

The structure $\mathcal R$ enjoys the \textit{property $(a)$} if any $\myds$-subset $X$ of $R$ has an empty interior if $X\langle r,s \rangle$ have empty interiors for all $r>0$ and $s>0$.
Here, $\{X\langle r,s \rangle\}_{r>0,s>0}$ is an arbitrary $\myds$-family with $X=\bigcup_{r>0,s>0}X\langle r,s \rangle$.

The structure $\mathcal R$ enjoys the \textit{property $(b)$} if at least one of $\myds$-subsets $X_1$ and $X_2$ of $R$ has a nonempty interior when the union $X_1 \cup X_2$ has a nonempty interior.
\end{definition}

We give the first key lemma below.
We prove it using the techniques employed in \cite{DMS,Fuji4}.
\begin{lemma}\label{lem:key}
Let $\mathcal R=(R,<,+,0,\ldots)$ be a definably complete expansion of a densely linearly ordered abelian group enjoying the property $(a)$.
Let $X$ be a $\myds$-subset of $R^n$ and $\{X\langle r,s \rangle\}_{r>0,s>0}$ be a $\myds$-family with $X=\bigcup_{r>0,s>0}X\langle r,s \rangle$.
\begin{enumerate}
\item[$(A)_n$] The $\myds$-set $X$ has an empty interior if the CBD sets $X\langle r,s \rangle$ have empty interiors for all $r>0$ and $s>0$.
\item[$(B)_n$] Set 
\begin{equation*}
\mathcal I(X)=\{x \in R^{n-1}\;|\; \text{the fiber } X_x \text{ contains a nonempty open interval}\}\text{.}
\end{equation*}
It is a $\myds$-set.
If $\mathcal I(X)$ has a nonempty interior, the CBD set $X\langle r,s \rangle$ has a nonempty interior for some $r>0$ and $s>0$.
\end{enumerate}
\end{lemma}
\begin{proof}
We demonstrate $(A)_n$ and $(B)_n$ simultaneously by induction on $n$.
The assertion $(A)_1$ is identical to the property (a).
Therefore, the assertion $(A)_1$ holds true by the assumption.

We next demonstrate that the property $(a)$ and the assertion $(A)_n$ implies the assertion $(B)_n$ when $n \geq 1$.
Set $Y\langle r,s \rangle=\{x \in R^n\;|\; \exists t \in R,\ [t-s,t+s] \subset (X\langle r,s \rangle)_x\}$.
The set $Y\langle r,s \rangle$ is CBD.
We show that $\mathcal I(X)=\bigcup_{r,s}Y\langle r,s \rangle$.
It means that $\mathcal I(X)$ is a $\myds$-set.
The inclusion $\bigcup_{r,s}Y\langle r,s \rangle \subseteq \mathcal I(X)$ is obvious.
We demonstrate the opposite inclusion.
Take an arbitrary point $x \in \mathcal I(X)$.
We have $\myint(X_x) \not= \emptyset$.
We get $\myint((X\langle r,s \rangle)_x) \not=\emptyset$ for some $r>0$ and $s>0$ by the property $(a)$.
There exist $t \in R$ and $s'>0$ with $[t-s',t+s'] \subseteq (X\langle r,s \rangle)_x$.
Replacing $s$ with $\min\{s,s'\}$, we may assume that $[t-s,t+s] \subseteq (X\langle r,s \rangle)_x$.
It means that $x \in Y_{r,s}$.

By the assertion $(A)_n$, the CBD set $Y\langle r,s \rangle$ has a nonempty interior for some $r>0$ and $s>0$.
The CBD set $X\langle r,s \rangle$ has a nonempty interior by Lemma \ref{lem:interior0}.
We have demonstrated the assertion $(B)_n$.

We next prove that the assertion $(B)_n$ implies the assertion $(A)_{n+1}$ when $n \geq 1$.
When the $\myds$-subset $X$ of $R^{n+1}$ has a nonempty interior, the set $\mathcal I(X)$ has a nonempty interior because $X$ contains an open box.
For some $r>0$ and $s>0$, the CBD set $X\langle r,s \rangle$ has a nonempty interior by the assumption.
\end{proof}

We give four corollaries of this lemma.

\begin{corollary}\label{cor:key1}
Let $\mathcal R$ and $X$ be as in Lemma \ref{lem:key}.
If the $\myds$-set $\mathcal I(X)$ has a nonempty interior, then $X$ has a nonempty interior.
\end{corollary}
\begin{proof}
Obvious by Lemma \ref{lem:key}$(B)_n$.
\end{proof}

\begin{corollary}\label{cor:key3}
Let $\mathcal R$ and $X$ be as in Lemma \ref{lem:key}.
Let $\pi:R^n \rightarrow R^d$ be a coordinate projection such that $\pi(X)$ has a nonempty interior.
Then, there exists a nonempty bounded open box $B$ in $R^n$ such that, for all $x \in \pi(B)$, the fibers $\pi^{-1}(x) \cap (B \cap X)$ of $B \cap X$ are not empty sets.
\end{corollary}
\begin{proof}
We may assume that the projection $\pi$ is the projection onto the first $d$ coordinates without loss of generality.
Let $\{X\langle r,s \rangle\}_{r>0,s>0}$ be a $\myds$-family with $X=\bigcup_{r,s}X\langle r,s \rangle$.
Since $\pi(X)$ has a nonempty interior, the set $\pi(X\langle r,s \rangle)$ has a nonempty interior for some $r>0$ and $s>0$ by Lemma \ref{lem:key}$(A)_d$.
Take a nonempty bounded open box $U$ in $R^d$ contained in $\pi(X\langle r,s \rangle)$.
Since $X\langle r,s \rangle$ is bounded, there exists a nonempty bounded open box $V$ in $R^{n-d}$ such that $X\langle r,s \rangle \cap (U \times R^{n-d}) \subseteq U \times V$.
Set $B=U \times V$.
We have $\pi^{-1}(x) \cap (B \cap X\langle r,s \rangle) \neq \emptyset$ for all $x \in U$.
In particular, we obtain $\pi^{-1}(x) \cap (B \cap X) \neq \emptyset$.
\end{proof}

\begin{corollary}\label{cor:key2}
Let $\mathcal R=(R,<,+,0,\ldots)$ be a definably complete expansion of a densely linearly ordered abelian group enjoying the properties $(a)$ and $(b)$.
Let $X_1$ and $X_2$ be $\myds$-subsets of $R^n$.
If $X_1 \cup X_2$ has a nonempty interior, then at least one of $X_1$ and $X_2$ has a nonempty interior.
\end{corollary}
\begin{proof}
We show the corollary by induction on $n$.
The corollary follows from the property $(b)$ when $n=1$.
Consider the case in which $n>1$.
Set $X=X_1 \cup X_2$.
We have $\mathcal I(X)=\mathcal I(X_1) \cup \mathcal I(X_2)$ by the property $(b)$. 
Since $X$ has a nonempty interior, $\mathcal I(X)$ also has a nonempty interior by the definition of $\mathcal I(X)$.
The sets $\mathcal I(X_1)$ and $\mathcal I(X_2)$ are $\myds$ by Lemma \ref{lem:key}$(B)_n$.
By the induction hypothesis, at least one of $\mathcal I(X_1)$ and $\mathcal I(X_2)$ has a nonempty interior.
At least one of $X_1$ and $X_2$ has a nonempty interior by Corollary \ref{cor:key1}.
\end{proof}

\begin{corollary}\label{cor:key4}
Let $\mathcal R=(R,<,+,0,\ldots)$ be as in Corollary \ref{cor:key2}.
A definable constructible set $X$ has a nonempty interior when the closure of $X$ has a nonempty interior.
\end{corollary}
\begin{proof}
Note that every constructible definable set is a finite boolean combination of open definable sets by \cite{DM}.
Therefore, the definable constructible set $X$ is the union of a finite locally closed definable sets $X_1, \ldots X_n$.
Assume that the closure $\mycl(X)$ contains a nonempty interior.
Since $\mycl(X)=\bigcup_{i=1}^n \mycl(X_i)$ and definable closed sets are $\myds$, the closure $\mycl(X_i)$ has a nonempty interior for some $1 \leq i \leq n$ by Corollary \ref{cor:key2}.
A standard topological argument shows that $X_i$ has a nonempty interior because $X_i$ is locally closed.
It implies that $X$ has a nonempty interior.
\end{proof}

\subsection{Second key lemma}
We next prove the second key lemma.
A parameterized family of definable sets $\{X\langle r \rangle\}_{r>0}$ is called \textit{increasing} if $X\langle r \rangle \subseteq X\langle r' \rangle$ whenever $r \leq r'$.
It is called \textit{decreasing} if $X\langle r' \rangle \subseteq X\langle r \rangle$ whenever $r \leq r'$.
\begin{lemma}\label{lem:basic1}
Let $\mathcal R=(R,<,+,0,\ldots)$ be a definably complete expansion of a densely linearly ordered abelian group such that an arbitrary definable closed subset of $R$ is either  discrete or contains a nonempty open interval.
Let $\{X\langle r \rangle\}_{r>0}$ be a parameterized family of definable, discrete and  closed subsets of $R$.
Assume that it is either increasing or decreasing.
Then the union $X=\bigcup_{r>0} X\langle r \rangle$ is discrete and closed.  
\end{lemma}
\begin{proof}
We prove this lemma similarly to \cite[Theorem 4.3]{Fuji3}.
We only consider the case in which the given parameterized family is increasing.
We can prove the lemma similarly in the other case.

Note that a definable subset of a definable, closed and discrete subset is again discrete and closed.
In particular, under our assumption, the definable set $X$ is discrete and closed unless the closure $\mycl(X)$ contains a nonempty open interval.
Therefore, we have only to lead to a contradiction assuming that the closure $\mycl(X)$ contains a nonempty open interval $I$.

Take a point $a \in X$ which is also contained in the open interval $I$.
Consider the definable function $f:\{r \in R\;|\;r >0\} \rightarrow \{x \in R\;|\; x  > a\}$ defined by $f(r)= \inf \{x >a \;|\; x \in X{\langle r\rangle}\}$.
It is obvious that $f$ is a decreasing function because $\{X{\langle r\rangle}\}_{r>0}$ is increasing. 

We demonstrate that the image $\myim(f)$ of the function $f$ is discrete and closed.
As previously demonstrated, if the closure $\mycl(\myim(f))$ is discrete and closed, the image $\myim(f)$ is also discrete and closed.
We assume that $\mycl(\myim(f))$ is not discrete and closed for contradiction.
A nonempty open interval $J$ is contained in $\mycl(\myim(f))$ by the assumption of the lemma.
Take a point $b \in \myim(f) \cap J$ and a point $r>0$ with $b=f(r)$.
Since $X{\langle r\rangle}$ is closed, we have $b \in X{\langle r\rangle}$.
Any point $b' \in \myim(f)$ with $b' > b$ is also contained in $X{\langle r\rangle}$.
In fact, take a point $r'>0$ with $b'=f(r')$.
If $r'>r$, the set $X{\langle r'\rangle}$ contains the point $b$ because $X{\langle r\rangle} \subseteq X{\langle r'\rangle}$.
Then we have $b'=f(r') \leq b$ by the definition of the function $f$, and this is a contradiction.
If $r'<r$, we have $b' \in X{\langle r'\rangle} \subseteq X{\langle r\rangle}$.

Set $b_1 = \inf\{b' \in \myim(f)\;|\;b'>b\} $.
We have $b_1 \in X{\langle r\rangle}$ and $b_1> b$ because $\{b' \in \myim(f)\;|\;b'>b\} \subseteq X{\langle r\rangle}$ and $X{\langle r\rangle}$ is closed and discrete.
The open interval $(b,b_1)$ has an empty intersection with $\myim(f)$.
It contradicts the assumption that $\mycl(\myim(f))$ contains the open interval $J$ with $b \in J$.
We have demonstrated that $\myim(f)$ is discrete and closed.

Let $c$ be the infimum of $\myim(f)$.
We have $c>a$ and $c \in \myim(f)$ because $\myim(f)$ is discrete and closed.
On the other hand, there exists $b \in X$ with $a<b<c$ because the closure $\mycl(X)$ contains the open interval $I$ with $a \in I$.
Take a point $r>0$ with $b \in X{\langle r\rangle}$.
By the definition of $f$, we have $a < f(r) \leq b<c$. 
It contradicts the definition of $c$.
\end{proof}

We provide corollaries of Lemma \ref{lem:basic1}.

\begin{corollary}\label{cor:basic1}
Let $\mathcal R$ be as in Lemma \ref{lem:basic1}.
The structure $\mathcal R$ possesses the property $(a)$ in Definition \ref{def:technical}.
Furthermore, an arbitrary $\myds$-subset of $R$ is either discrete and closed or contains a nonempty open interval.
\end{corollary}
\begin{proof}
Let $X$ be a $\myds$-subset of $R$ and $\{X\langle r,s \rangle\}_{r>0,s>0}$ be a $\myds$-family with $X=\bigcup_{r,s}X\langle r,s \rangle$.
If the CBD set $X\langle r,s \rangle$ contains a nonempty open interval for some $r>0$ and $s>0$, the union $X$ also contain an open interval.
We have only to show that $X$ is discrete and closed if the CBD sets $X\langle r,s \rangle$ do not contain an open interval for all $r>0$ and $s>0$.
The CBD sets $X\langle r,s \rangle$ are discrete and closed by the assumption.
When $r$ is fixed, the parameterized family $\{X\langle r,s \rangle\}_{s>0}$ is decreasing.
The union $X\langle r \rangle=\bigcup_{s>0}X\langle r,s \rangle$ is discrete and closed by Lemma \ref{lem:basic1}.
Apply Lemma \ref{lem:basic1} again to the increasing parameterized family $\{X\langle r \rangle\}_{r>0}$.
The $\myds$-set $X=\bigcup_{r>0}X\langle r \rangle$ is discrete and closed.
\end{proof}

\begin{corollary}\label{cor:basic2}
Let $\mathcal R$ be as in Lemma \ref{lem:basic1}.
The structure $\mathcal R$ possesses the property $(b)$ in Definition \ref{def:technical}.
\end{corollary}
\begin{proof}
Immediate by the `furthermore' part of Corollary \ref{cor:basic1}.
\end{proof}

Using the previous corollaries, we get the following lemma:

\begin{lemma}\label{lem:good_point_pre}
Let $\mathcal R$ be as in Lemma \ref{lem:basic1} and $X$ be a $\myds$-subset of $R^n$ of dimension $d$.
Take a coordinate projection $\pi:X \rightarrow R^d$ such that $\pi(X)$ has a nonempty interior.
Then, there exists a $\myds$-subset $Z$ of $R^d$ such that $Z$ has an empty interior and the fiber $X \cap \pi^{-1}(x)$ is discrete and closed for any $x \in \pi(X) \setminus Z$.
\end{lemma}
\begin{proof}
This lemma is proven in the same manner as \cite[Lemma 5.6]{Fuji4}.
We give a proof here for readers' convenience.
The assumptions of the corollaries in Section \ref{sec:keysub} are satisfied thanks to Corollary \ref{cor:basic1} and Corollary \ref{cor:basic2}.

For all $1 \leq i \leq n-d$, we can take coordinate projections $\pi_i:R^{n-i+1} \rightarrow R^{n-i}$ with $\pi = \pi_{n-d} \circ \cdots \circ \pi_1$.
We may assume that $\pi_i$ are the coordinate projections forgetting the last coordinate without loss of generality.
Set $\Pi_i=\pi_i \circ \cdots \circ \pi_1$ and $\Phi_i=\pi_{n-d} \circ \cdots \circ \pi_{i+1}$.
Consider the sets $T_i = \{x \in R^{n-i}\;|\; \pi_i^{-1}(x) \cap \Pi_{i-1}(X) \text{ contains a nonempty open interval}\}$.
The sets $T_i$ are $\myds$ and we have $\dim(T_i) < \dim \Pi_{i-1}(X) = \dim X = d$ by Corollary \ref{cor:key1}.
Set $U_i=\Phi_i(T_i) \subseteq R^d$ for all $1 \leq i \leq n-d$.
The projection images $U_i$ are $\myds$-sets.
We get $\myint(U_i) = \emptyset$ because $\dim(T_i) < d$.
Set $Z=\bigcup_{i=1}^{n-d} U_i$.
It also has an empty interior by Corollary \ref{cor:key2}.

The fiber $X \cap \pi^{-1}(x)$ is discrete and closed for any $x \in \pi(X) \setminus Z$.
In fact, let $y \in R^n$ be an arbitrary point with $x=\pi(y)$.
Set $y_0=y$ and $y_i=\Pi_i(y)$ for $1 \leq i \leq n-d$.
We have $y_{n-d}=x$ by the definition.
We construct an open box $B_i$ in $R^{n-d-i}$ for $0 \leq i \leq n-d$ such that $y_i \in B_i$ and $(\{x\} \times B_i) \cap \Pi_i(X)$ consists of at most one point in decreasing order.
When $i=n-d$, the open box $B_{n-d}=R^0$.
When $(\{x\} \times B_i) \cap \Pi_i(X)=\emptyset$, set $B_{i-1}=B_i \times R$.
We have $(\{x\} \times B_{i-1}) \cap \Pi_{i-1}(X)=\emptyset$.
When $(\{x\} \times B_i) \cap \Pi_i(X)\not=\emptyset$, the fiber $\Pi_{i-1}(X) \cap \pi_i^{-1}(y_i)$ is discrete and closed by Corollary \ref{cor:basic1}.
Therefore, there exists an open box $B_{i-1}$ in $R^{n-d+1-i}$ such that $\pi_i(B_{i-1})=B_i$, $y_{i-1} \in B_{i-1}$ and $(\{x\} \times B_{i-1}) \cap \Pi_{i-1}(X)$ consists of at most one point.
We have constructed the open boxes $B_i$ in $R^{n-d-i}$ for all $0 \leq i \leq n-d$.
The existence of $B_0$ implies that $X \cap \pi^{-1}(x)$ is discrete and closed.
\end{proof}

\section{Proof of Theorem \ref{thm:main}}\label{sec:proof}
Preparation has been done.
We finally get the following theorem:
\begin{theorem}\label{thm:constructible}
Let $\mathcal R=(R,<,+,0,\ldots)$ be a definably complete expansion of a densely linearly ordered abelian group such that an arbitrary definable closed subset of $R$ is either  discrete or contains a nonempty open interval.
A $\myds$-set is constructible.
\end{theorem}
\begin{proof}
We prove the theorem using the technique employed  in \cite{F}.
Note that the assumptions of the lemma and the corollaries in Section \ref{sec:keysub} are satisfied thanks to Corollary \ref{cor:basic1} and Corollary \ref{cor:basic2}.
We use this fact without notice in the proof.

Let $X$ be a $\myds$-subset of $R$ and $\{X\langle r,s \rangle\}_{r>0,s>0}$ be a $\myds$-family with $X=\bigcup_{r,s}X\langle r,s \rangle$.
We demonstrate the theorem by induction on the full dimension $\myfdim(X)$ of $X$.
Set $d=\dim X$.
We first consider the case in which $d=0$.
Let $\pi_i:R^n \rightarrow R$ be the projections onto the $i$-th coordinate for all $1 \leq i \leq n$.
The projection images $\pi_i(X)$ are discrete and closed for all $1 \leq i \leq n$ by Corollary \ref{cor:basic1} because $\pi_i(X)$ do not contain an open interval by the definition of dimension.
The set $X$ is a subset of the discrete and closed set $\prod_{i=1}^n \pi_i(X)$.
Therefore, it is also closed and discrete.
In particular, it is constructible.

Let us consider the case in which $d>0$.
There exists a coordinate projection $\pi:R^n \rightarrow R^d$ such that the image $\pi(X)$ has a nonempty interior.
We can take a $\myds$-subset $Z$ of $R^d$ such that $\myint(Z)=\emptyset$ and the fiber $X \cap \pi^{-1}(x)$ is discrete and closed for any $x \in \pi(X) \setminus Z$ by Lemma \ref{lem:good_point_pre}.
It is obvious that $\myfdim(\pi^{-1}(Z) \cap X)<\myfdim X$. 
The $\myds$-set $\pi^{-1}(Z) \cap X$ is constructible by the induction hypothesis.
We have only to demonstrate that $X \setminus (\pi^{-1}(Z) \cap X)$ is constructible.
Therefore, we may assume that $\pi^{-1}(x)$ is discrete and closed for any $x \in \pi(X)$.
 
The notation $\mylc(X)$ denotes the set of points having open boxes $U$ containing the points such that the intersections $U \cap X$ are locally closed.
The set $\mylc(X)$ is definable and locally closed by its definition.
We show that $\pi(Y)$ has an empty interior, where $Y=X \setminus \mylc(X)$.
Once it is proven, the theorem immediately follows from the induction hypothesis.
For contradiction, we assume that $\pi(Y)$ has a nonempty interior.
By Corollary \ref{cor:key3}, there exists a bounded open box $B$ such that, 
\begin{description}
\item[$(*)$] for any $u \in \pi(B)$, the fiber $Y \cap B \cap \pi^{-1}(u)$ of $Y \cap B$ is not an empty set.
\end{description}

We next demonstrate that, for any $u \in \pi(B)$, there exist $r_u>0$ and $s_u>0$ such that $T_u:=X \cap U \cap \pi^{-1}(u)$ is contained in the CBD set $X\langle r_u,s_u \rangle$.
We fix $u \in \pi(B)$.
Set $C \langle r,s \rangle = T_u \setminus  X\langle r,s \rangle$ for all $r>0$ and $s>0$.
We have only to show that $C \langle r,s \rangle$ is an empty set for some $r>0$ and $s>0$.
Assume the contrary.
Since $T_u$ is discrete and closed, the set $C \langle r,s \rangle$ is also discrete and closed.
Since the parameterized family $\{C \langle r,s \rangle\}_{r>0,s>0}$ is a definable filtered collection of nonempty definable closed subsets of the bounded closed definable set $T_u$, the intersection $\bigcap_{r>0,s>0} C \langle r,s \rangle$ is not an empty set by Proposition \ref{prop:def_compact}.
It contradicts the equality $X=\bigcup_{r>0,s>0} X \langle r,s \rangle$.

Set 
\begin{align*}
& D\langle r,s \rangle= (X \cap B) \setminus X\langle r,s \rangle,\\
& V\langle r,s \rangle = \pi(B) \setminus \pi(D\langle r,s \rangle) \text{ and }\\
&W\langle r,s \rangle = \mycl(V\langle r,s \rangle).
\end{align*}
Note that the $\myds$-set $\pi(D\langle r,s \rangle)$ is constructible by the induction hypothesis.
The set $V\langle r,s \rangle$ is also constructible and its closure $W\langle r,s \rangle$ is a CBD set for any $r>0$ and $s>0$.
The existence of $r_u>0$ and $s_u>0$ for any $u \in \pi(B)$ implies the equality $\pi(B)=\bigcup_{r>0,s>0} V \langle r,s \rangle$.
In particular, the parameterized family $\{W\langle r,s \rangle\}_{r>0,s>0}$ of CBD sets is a $\myds$-family and the $\myds$-set $W=\bigcup_{r>0,s>0} W \langle r,s \rangle$ contains an open box $\pi(B)$.
Since $W$ has a nonempty interior, the CBD set $W \langle \widetilde{r},\widetilde{s} \rangle$ has a nonempty interior for some $\widetilde{r}>0$ and $\widetilde{s}>0$ by Lemma \ref{lem:key}$(A)_d$.
Since $V\langle \widetilde{r},\widetilde{s} \rangle$ is constructible, it also has a nonempty interior by Corollary \ref{cor:key4}.
By shrinking $\pi(B)$ if necessary, we may assume that $B$ satisfies the condition $(*)$ and $\pi(B)$ is contained in $V\langle \widetilde{r},\widetilde{s} \rangle$.
It means that $X \cap B$ is contained in $X\langle \widetilde{r},\widetilde{s} \rangle$.
We have $X \cap B=X\langle \widetilde{r},\widetilde{s} \rangle \cap B$, which implies that $X \cap B$ is locally closed because $X\langle \widetilde{r},\widetilde{s} \rangle$ is closed.
On the other hand, the condition $(*)$ implies that $X\cap B$ has a nonempty intersection with $Y$.
It contradicts the definition of $Y$.
\end{proof}

We demonstrate Theorem \ref{thm:main} below.

\begin{proof}[Proof of Theorem \ref{thm:main}]
$(1) \Rightarrow (2)$:
Lemma \ref{lem:quo} and Theorem \ref{thm:constructible} assert that a set definable in the original structure $\mathcal R$ is definable in the open core if and only if it is $\myds$.
By Corollary \ref{cor:basic1}, any subset of $R$ definable in the open core is either discrete and closed or contains a nonempty open interval.
The open core is locally o-minimal by Lemma \ref{lem:local}.

$(2) \Rightarrow (1)$: 
Immediate by Lemma \ref{lem:local} because closed sets definable in $\mathcal R$ are definable in the open core.

The `furthermore' part follows from Theorem \ref{thm:constructible}.
\end{proof}

\end{document}